\documentclass[a4paper,11pt]{article}

\usepackage[top=3.0cm,bottom=3.0cm,left=2.5cm,right=2.5cm]{geometry}

\usepackage{setspace}
\usepackage{lmodern}

\usepackage{graphicx}
\usepackage{amsfonts}
\usepackage{amssymb}
\usepackage{amsmath}
\usepackage{amsthm}
\usepackage[english]{babel}

\usepackage{hyperref}
\usepackage{cases}
\usepackage{authblk}
\usepackage{mathtools}
\usepackage{graphicx}
\usepackage{color}
\usepackage{ifpdf}
\usepackage{subfigure}
\usepackage[graphicx]{realboxes}

\newtheorem{lemma}{Lemma}[section]
\newtheorem{theorem}[lemma]{Theorem}
\newtheorem{corollary}[lemma]{Corollary}

\newtheorem{claim}[]{\noindent Claim}

\begin{document}

\setstretch{1.35} 
	
\title{Perfect divisions in ($P_2 \cup P_4$, bull)-free graphs}

{\author{Lizhong Chen\footnote{Email: lchendh@connect.ust.hk} \quad Hongyang Wang\footnote{Corresponding Author. Email: hwanghj@connect.ust.hk}}
			
{\affil{Department of Mathematics,\\ Hong Kong University of Science and Technology, \\Clear Water Bay, Hong Kong}

\date{\today}
\maketitle

\begin{abstract}
A graph $G$ has a perfect division if its vertex set can be partitioned into two sets $A$, $B$ such that $G[A]$ is perfect and $\omega(G[B]) < \omega(G)$. We call $G$ perfectly divisible if every induced subgraph of $G$ admits a perfect division. We prove that every ($P_2 \cup P_4$, bull)-free graph $G$ with $\omega(G) \geq 3$ has a perfect division if $G$ contains no homogeneous set. The clique-number condition is tight: a counterexample exists for $\omega(G) = 2$. Additionally, we present a short proof of the perfect divisibility of ($P_5$, bull)-free graphs, originally established by Chudnovsky and Sivaraman~\cite{CHUDPD2D}.
\end{abstract}

\textbf{Mathematics Subject Classification}: 05C15, 05C17, 05C69
				
\textbf{Keywords}: Graph colouring; Perfect divisibility; Bull-free graphs; $P_2\cup P_4$-free graphs.

\section{Introduction}

All graphs considered in this paper are finite and simple. Let $G$ be a graph with vertex set $V(G)$ and edge set $E(G)$. Two vertices $u,v\in V(G)$ are adjacent if and only if $uv \in E(G)$. We denote by $u\sim v$ if $u,v$ are adjacent and $u\nsim v$ if otherwise. Let $d(v)$ denote the number of vertices adjacent to $v$. The complement $G^c$ of $G$ is the graph with vertex set $V(G)$ and $uv \in G^c$ if and only if $uv \notin G$. For two graphs $H$ and $G$, $H$ is an \emph{induced subgraph} of $G$ if $V(H)\subseteq V(G)$, and $uv \in E(H)$ if and only if $uv \in E(G)$. If $G$ has an induced subgraph isomorphic to $H$, then we say that $G$ \emph{contains} $H$, otherwise $G$ is said to be $H$-\emph{free}. For a family of graphs $\{H_1,H_2,\ldots\}$, we say that $G$ is $(H_1,H_2,\ldots)$-free if G is $H$-free for every $H\in\{H_1,H_2,\ldots\}$. We denote by $G[X]$ the induced subgraph of $G$ with vertex set $X\subseteq V(G)$. For two vertex-disjoint graph $G_1$ and $G_2$, the \emph{union} $G_1\cup G_2$ is the graph with $V(G_1\cup G_2)=V(G_1)\cup V(G_2)$ and $E(G_1\cup G_2)=E(G_1)\cup E(G_2)$.

A clique $K_n$ is a graph on $n$ vertices such that every two vertices of $K_n$ are adjacent. For an integer $k\geq1$, $P_k$ and $C_k$ denote the path on $k$ vertices and the cycle on $k$ vertices, respectively. A \emph{path} in a graph is a sequence $p_1-\cdots-p_k$ of distinct vertices such that $p_i \sim p_j$ if and only if $|i-j|=1$. A \emph{cycle} in a graph is a sequence $c_1-\cdots-c_k-c_1$ of distinct vertices such that $c_i\sim c_j$ if and only if $|i-j|=\{1,k-1\}$. A \emph{hole} is an induced subgraph that is isomorphic to the cycle $C_k$ with $k\geq4$, and $k$ is the \emph{length} of the hole. A hole is \emph{odd} if $k$ is \emph{odd}, and \emph{even} otherwise. An \emph{antihole} in a graph is an induced subgraph that is isomorphic to $C^c_k$ with $k\geq4$, and $k$ is the length of the antihole. An antihole is \emph{odd} if $k$ is odd, and \emph{even} otherwise. The \emph{bull} is a graph with vertex set $\{x_1,x_2,x_3,y,z\}$ and edge set $\{x_1x_2,x_1x_3,x_2x_3,x_1y,x_2z\}$, i.e. a triangle with two disjoint pendant edges.

The \emph{clique number} of $G$, denoted by $\omega(G)$, is the number of vertices in a largest clique of $G$. A graph G is called \emph{k-colourable}, if its vertices can be coloured with $k$ colours such that adjacent vertices receive distinct colours. The smallest integer $k$ such that a given graph $G$ is $k$-colourable is called its \emph{chromatic number}, denoted by $\chi(G)$. A graph $G$ is \emph{perfect} if we have $\chi(H)=\omega(H)$ for every induced subgraph $H$ of $G$. Gy\'{a}rf\'{a}s \cite{GY} introduced the notion of $\chi$-boundedness as a natural extension of perfect graphs in 1970. A class of graphs $\mathcal{G}$ is $\chi$-\emph{bounded} if there exists a real-valued function $f$ such that $\chi(G)\leq f(\omega(G))$ holds for every graph $G\in \mathcal{G}$, and $f$ is called a \emph{binding function}. Clearly, the class of perfect graphs is $\chi$-$bounded$ with $f(x)=x$ as the $\chi$-binding function. 

\input{collection.tpx}

Ho\`ang~\cite{HOANGBANNER} introduced the concept of perfect divisibility in 2018. A graph $G$ has a \emph{perfect division} if its vertex set can be partitioned into sets $A$ and $B$ such that $G[A]$ is perfect and $\omega(G[B]) < \omega(G)$. $G$ is said to be \emph{perfectly divisible} if every induced subgraph admits a perfect division. Evidently, a graph is perfectly divisible if it is perfect. A nice feature of perfect divisibility as pointed out by Chudnovsky and Sivaraman~\cite{CHUDPD2D} is that we can have a quadratic binding function $\binom{\omega(G)+1}{2}$ for a perfectly divisible graph. 

Recently, there has been much research on the perfect divisibility of graphs. For instance, Ho\`ang~\cite{HOANGBANNER} proved the perfect divisibility of (banner, odd hole)-free graphs and conjectured that all odd hole-free graphs are perfectly divisible; Chudnovsky and Sivaraman~\cite{CHUDPD2D} proved the perfect divisibility of (odd hole, bull)-free graphs and ($P_5$, bull)-free graphs; Deng and Chang~\cite{DENG} proved that every ($P_2\cup P_3$, bull)-free graph $G$ with $\omega(G)\geq3$ has a perfect division if $G$ admits no homogeneous set.

In this paper, we prove that every ($P_2\cup P_4$, bull)-free graph with $\omega(G)\geq3$ has a perfect division if $G$ admits no homogeneous set, generalising the result of Deng and Chang~\cite{DENG}. 

\begin{theorem}\label{THM1}
    If $G$ is a $(P_2\cup P_4,\mbox{bull})$-free graph with $\omega(G)\geq3$, then either $G$ contains a homogeneous set or $G$ has a perfect division.
\end{theorem}

Furthermore, we give a short proof of the perfect divisibility of ($P_5$, bull)-free graphs, first established by Chudnovsky and Sivaraman~\cite{CHUDPD2D}.

\begin{theorem}\cite{CHUDPD2D}\label{THM2}
    ($P_5$, bull)-free graphs are perfectly divisible.
\end{theorem}

\section{Preliminaries}

We set up the notation that will be used as follows: For a graph $G$ and a vertex $v\in V(G)$, the \emph{neighbourhood} of $v$, denoted by $N_G(v)$, is the set of all vertices adjacent to $v$, and the \emph{closed neighbourhood} $N_G[v]$ is defined as $N_G(v)\cup \{v\}$. They can be simplified to $N(v)$ and $N[v]$ when the graph $G$ is clear from context. Let $M_G(v)$ (or M(v)) be $V(G)\backslash N_G(v)$.

Let $A,B\subseteq V(G)$ be two disjoint sets of vertices. We say that $A$ is \emph{complete} to $B$ if every vertex of $A$ is adjacent to every vertex of $B$. Similarly,  $A$ is said to be \emph{anticomplete} to $B$ if every vertex of $A$ is nonadjacent to every vertex of $B$. A vertex set $X\subseteq V(G)$ is a \emph{homogeneous set} if $1<|X|<|V(G)|$ and every vertex in $V(G)\backslash X$ is either complete or anticomplete to $X$. A graph $G$ is \emph{minimally non-perfectly divisible} (MNPD for short) if $G$ is not perfectly divisible but each of its proper induced subgraph is. Hu, Xu and Zhang proposed that MNPD graphs admit no homogeneous set in 2025.

\begin{lemma}\cite{XUMNPD}\label{NOHOM}
    No MNPD graphs have a homogeneous set.
\end{lemma}

Let $H$ be an induced subgraph of $G$, and let $v \in V(G) \backslash V(H)$. The vertex $v$ is called a \emph{center} (respectively, \emph{anticenter}) of $H$ if it is complete (respectively, anticomplete) to $V(H)$. For a bull-free graph, Chudnovsky and Safra~\cite{CHUDHOM} proved the following result.

\begin{lemma}\cite{CHUDHOM}\label{BULLHOM}
If a bull-free graph $G$ contains an odd hole or an odd antihole with a center and an anticenter, then $G$ contains a homogeneous set.
\end{lemma}

By Lemma~\ref{NOHOM} and Lemma~\ref{BULLHOM}, we can deduce that

\begin{corollary}\label{CORO1}
A bull-free MNPD graph does not contain an odd hole or an odd antihole with a center and an anticenter.
\end{corollary}

The following lemma is used several times in the sequel. 

\begin{lemma}\label{LEM}
Let $G$ be a bull-free graph either admits no homogeneous set or is MNPD. If $G$ contains an odd antihole $X$ and an anticenter $v$ of $X$, then $\forall x\in N(v)$, $|N(x)\cap V(X)|\leq 2$. Furthermore, if $|N(x)\cap V(X)|= 2$, then $N(x)\cap V(X)$ are two nonadjacent vertices.
\end{lemma}

\begin{proof}
    Let $V(X)=\{v_1, v_2, \ldots, v_n\}$, and $v_i\sim v_j$ if and only if $|i-j|\neq1$ (indices are modulo $n$).

\begin{claim}\label{CLAIML1}
If $x\sim v_i,v_j$ such that $v_i\sim v_j$, then $x$ is complete to $X$.   
\end{claim}

Suppose that $x \sim v_1, v_k$ for some integer $k$ with $3 \leq k \leq n-1$. If $k \leq n-2$, then $x \sim v_{k+1}$; otherwise, the subgraph $G[\{v, x, v_1, v_k, v_{k+1}\}]$ is a bull. Similarly, if $k+1 \leq n-2$, then $x \sim v_{k+2}$, as otherwise $G[\{v, x, v_1, v_{k+1}, v_{k+2}\}]$ is a bull. Repeating this argument, it follows that $x$ is complete to $\{v_k, \ldots, v_{n-1}\}$.

On the other hand, if $k \geq 4$, then $x \sim v_{k-1}$; otherwise, $G[\{v, x, v_1, v_k, v_{k-1}\}]$ would be a bull. If $k-1 \geq 4$, then $x \sim v_{k-2}$ since $G\{[v, x, v_1, v_{k-1}, v_{k-2}]\}$ cannot be a bull. Proceeding inductively, we conclude that $x$ is complete to $\{v_3, \ldots, v_{k}\}$.

Now we consider the two induced subgraphs $G[\{v, x, v_1, v_{n-1}, v_2\}]$ and $G[\{v, x, v_1, v_3, v_n\}]$. Since each of them cannot be a bull, we conclude that $x\sim v_2, v_n$. This proves Claim \ref{CLAIML1}.

Assume for contradiction that $|N(x)\cap V(X)|\geq 3$, or $N(x)\cap V(X)$ are exactly two adjacent vertices. Then, by Claim~\ref{CLAIML1}, $x$ is complete to $X$. This makes $G$ contain an odd antihole $X$ with a center $x$ and an anticenter $v$. By Lemma~\ref{BULLHOM} and Corollary~\ref{CORO1}, $G$ contains a homogeneous set and cannot be an MNPD graph, a contradiction. This proves Lemma \ref{LEM}.
\end{proof}

\section{The Structure of ($P_2 \cup P_4$, bull)-Free Graphs}

Notice that a $(P_2\cup P_4,\mbox{bull})$-free graph $G$ with $\omega(G)\geq3$ may not be perfectly divisible in general. In fact, $G$ cannot be perfectly divisible if $G$ contains a Gr\"otzsch graph (see Figure 2.) since Gr\"otzsch graph admits no perfect division. Let $H$ be the Gr\"otzsch graph in $G$. If $H$ has a perfect division, then we can partition $V(H)$ into two sets $A$ and $B$ such that $H[A]$ is perfect and $\omega(H[B])<\omega(H)$. It is well known that $\chi(H)=4$ and $\omega(H)=2$. Thus, $H[A]$ is 2-colourable since it is perfect, and $H[B]$ is 1-colourable by $\omega(H[B])<\omega(H)$. Hence $\chi(H)\leq 3$, a contradiction.

\input{gro.tpx}

\begin{proof}[\textnormal{\textbf{Proof of Theorem \ref{THM1}}}]

Let $G$ be a ($P_2 \cup P_4$, bull)-free graph with $\omega(G) \geq 3$. Assume that $G$ contains no homogeneous set. Then $G$ is connected; otherwise, there exists a component that is a homogeneous set. It suffices to prove that there is a vertex $v\in V(G)$ such that $G[M(v)]$ is perfect, for then $\{v\}\cup M(v)$ and $N(v)$ form a perfect division of $G$. Let $A\coloneq \{u\in V(G)| N(u) \mbox{ contains a clique on } \omega(G) -1 \mbox{ vertices} \}$, and let $v$ be a vertex of maximum degree in $A$. We will prove that $G[M(v)]$ is  perfect.

For the sake of contradiction, suppose $G[M(v)]$ is not perfect. By the Strong Perfect Graph Theorem~\cite{SPGT}, $G[M(v)]$ contains at least one of the following: an odd antihole of length at least seven, a $C_5$, or a $C_7$.

\textbf{Case 1.} $G[M(v)]$ contains an odd antihole $X$ of length at least seven.

Let $V(X)=\{v_1, v_2, \ldots, v_n\}$, and $v_i\sim v_j$ if and only if $|i-j|\neq1$ (indices are modulo $n$). $N(v)\neq \emptyset$ since $G$ is connected. Let $x\in N(v)$, $|N(x)\cap X|\leq 2$ by Lemma \ref{LEM}. If $N(x)\cap V(X)= \emptyset$, then for any $v_i\in V(X)$, $G[\{v,x,v_i,v_{i+2},v_{i-1},v_{i+1}\}]$ is a $P_2\cup P_4$ , a contradiction. If $N(x)\cap V(X) = \{v_i\}$, then $G[\{v,x,v_{i+2},v_{i+4},v_{i+1},v_{i+3}\}]$ is a $P_2\cup P_4$, a contradiction. By Lemma \ref{LEM}, we may assume that $N(x)\cap V(X) = \{v_i, v_{i+1}\}$. But now $G[\{v,x,v_{i+3},v_{i+5},v_{i+2},v_{i+4}\}]$ is a $P_2\cup P_4$, a contradiction.

\textbf{Case 2.} $G[M(v)]$ contains a $C_5$.

Let $V(C_5)=\{v_1, v_2, v_3, v_4, v_5\}$, and $v_i\sim v_j$ if and only if $|i-j|=1$ (indices are modulo $5$).
Since $C_5$ is an odd antihole of length 5, it follows that $|N(x)\cap V(C_5)|\leq 2$. If $N(x)\cap V(C_5)=\emptyset$. then $G[\{v,x,v_1,v_2,v_3,v_4\}]$ is a $P_2\cup P_4$, a contradiction. If $N(x)\cap X = \{v_i\}$, then $G[\{v,x,v_{i+1},v_{i+2},v_{i+3},v_{i+4}\}]$ is a $P_2\cup P_4$, a contradiction. Thus, each vertex in $N(v)$ has exactly two nonadjacent neighbours in $V(C_5)$ by Lemma \ref{LEM}.

\begin{claim}\label{CLAIMT1}
For any two vertices $x,y\in N(v)$, if $x\sim y$, then $N(x)\cap V(C_5) = N(y)\cap V(C_5)$.    
\end{claim}
We first suppose $N(x)\cap X=\{v_1, v_3\}$. Consider the case where $x$ and $y$ have no common neighbours in $V(C_5)$. By Lemma \ref{LEM}, $N(y)\cap V(C_5) = \{v_2,v_4\}$ or $N(y)\cap V(C_5) = \{v_2,v_5\}$. But now $G[\{v,x,y,v_1,v_4\}]$ is a bull if  $N(y)\cap V(C_5) = \{v_2,v_4\}$, and  $G[\{v,x,y,v_3,v_5\}]$ is a bull if $N(y)\cap V(C_5) = \{v_2,v_5\}$. Now suppose $x$ and $y$ have exactly one common neighbours in $V(C_5)$, then Lemma \ref{LEM} implies that either $N(y)\cap V(C_5) = \{v_1,v_4\}$ or $N(y)\cap V(C_5) = \{v_3,v_5\}$. But now $G[\{x,y,v_1,v_2,v_4\}]$ is a bull if $N(y)\cap V(C_5) = \{v_1,v_4\}$, and $G[\{x,y,v_2,v_3,v_5\}]$ is a bull if $N(y)\cap V(C_5) = \{v_3,v_5\}$. Hence, $N(x)\cap V(C_5) = N(y)\cap V(C_5)$. This proves Claim \ref{CLAIMT1}.

Since $\omega(G) \geq 3$ and $v \in A$, there exists a clique of order $\omega(G)-1 \geq 2$ in $G[N(v)]$. Let $x$ and $y$ be vertices in this clique. By Claim~\ref{CLAIMT1}, we may assume $N(x) \cap V(C_5) = N(y) \cap V(C_5) = \{v_1, v_3\}$. Consequently, both $v_1$ and $v_3$ are in $A$.

Note that $v$ is chosen as a vertex of maximum degree in $A$, there must exist a vertex $z \in N(v)$ such that $z \nsim v_1$; otherwise, we would have $d(v_1) > d(v)$, a contradiction. By Claim \ref{CLAIMT1}, $z$ is anticomplete to $\{x,y\}$. Now suppose $z\sim v_2$. Lemma \ref{LEM} implies that either $N(z)\cap V(C_5) = \{v_2,v_4\}$ or $N(z)\cap V(C_5) = \{v_2,v_5\}$. In the first case, $\{x,y,v_2,z,v_4,v_5\}$ induces a $P_2 \cup P_4$; in the second case, so does $\{x,y,v_2,z,v_5,v_4\}$. Therefore, $z \nsim v_2$. Since $z \nsim v_1$, it follows that $N(z) \cap V(C_5) = \{v_3, v_5\}$. 

The argument in the previous paragraph further shows that if a vertex in $N(v)$ is nonadjacent to $v_1$, it must be adjacent to $v_3$. Since $v$ has maximum degree in $A$, there exists a vertex $z^\prime \in N(v)$ such that $z^\prime \nsim v_3$; otherwise $d(v_3) > d(v)$. Then Claim \ref{CLAIMT1} implies that $z^\prime$ is anticomplete to $\{x,y,z\}$. Applying a similar reasoning to $z^\prime$ as to $z$, $N(z^\prime) \cap V(C_5) = \{v_1, v_4\}$. But then $G[\{x, y, z, v_5, v_4, z^\prime\}]$ is a $P_2\cup P_4$, a contradiction.

\textbf{Case 3.} $G[M(v)]$ contains a $C_7$.

Let $V(C_7)=\{v_1, v_2, \ldots, v_7\}$, and $v_i\sim v_j$ if and only if $|i-j|=1$ (indices are modulo $7$).

\begin{claim}\label{CLAIMT2}
    $\forall x\in N(v)$, $N(x)\cap V(C_7)$ are three pairwise nonadjacent vertices.
\end{claim}

Let $x\in N(v)$. Then $G$ is $P_2\cup P_4$-free implies that $|N(v)\cap V(C_7)|\geq 3$. If $4\leq |N(v)\cap V(C_7)|\leq 6$, there must exist an integer $i\in \{1,2,\ldots,7\}$ such that $x$ is complete to $\{v_i,v_{i+1}\}$, and $x\nsim v_{i-1}$ or $x\nsim v_{i+2}$. In the first case,  $\{v,x,v_{i-1},v_i, v_{i+1}\}$ induces a bull; in the second case, so does $\{v, x, v_i,v_{i+1},v_{i+2}\}$. If $|N(v)\cap V(C_7)| = 7$, i.e. $v$ is complete to $V(C_7)$. Then $x$ is a center of $C_7$ and $v$ is an anticenter of $C_7$. By Lemma~\ref{BULLHOM}, $G$ contains a homogeneous set, a contradiction. Thus, $|N(v)\cap V(C_7)|=3$. By the previous discussion in the case of $4\leq |N(v)\cap V(C_7)|\leq 6$, $x$ cannot be complete to two adjacent vertices in $V(C_7)$. Hence, we can deduce that $N(x)\cap V(C_7)$ consists of three pairwise nonadjacent vertices.

\begin{claim}\label{CLAIMT3}
For any two vertices $x,y\in N(v)$, if $x\sim y$, then $N(x)\cap V(C_7) = N(y)\cap V(C_7)$.    
\end{claim}

By Claim \ref{CLAIMT2}, without lost of generality, $N(x)\cap V(C_7) = \{v_1, v_3, v_5\}$. We first show that $y$ is anticomplete to $\{v_2, v_4\}$. It suffices to prove that $y \nsim v_2$, as the conclusion $y \nsim v_4$ then follows by symmetry. Suppose, to the contrary, that $y \sim v_2$. Then, by Claim~\ref{CLAIMT2}, we have $N(y) \cap V(C_7) = \{v_2, v_4, v_6\}$, $\{v_2, v_4, v_7\}$, or $\{v_2, v_5, v_7\}$. In the first case, $\{v, x, y, v_1, v_6\}$ induces a bull; in the second case, $\{v, x, y, v_5, v_7\}$ induces a bull; in the last case, $\{v, x, y, v_3, v_7\}$ induces a bull. Thus, we conclude that $y$ is anticomplete to $\{v_2, v_4\}$. Next, we show that $y$ is anticomplete to $\{v_6, v_7\}$. As before, it suffices to prove that $y \nsim v_6$, since the conclusion $y \nsim v_7$ then follows by symmetry. Assume, for contradiction, that $y \sim v_6$. Then, by Claim~\ref{CLAIMT2} and the fact that $y$ is anticomplete to $\{v_2, v_4\}$, it follows that $N(y) \cap V(C_7) = \{v_1, v_3, v_6\}$. But now $\{x, y, v_2, v_3, v_6\}$ induces a bull. Hence, $N(y) \cap V(C_7) = \{v_1, v_3, v_5\}$, completing the proof of Claim~\ref{CLAIMT3}.

Since $\omega(G) \geq 3$ and $v \in A$, there exists a clique of order $\omega(G)-1 \geq 2$ in $G[N(v)]$. Let $x$ and $y$ be vertices in this clique. By Claim~\ref{CLAIMT1}, we may assume $N(x) \cap V(C_7) = N(y) \cap V(C_7) = \{v_1, v_3, v_5\}$. Consequently, $v_1$, $v_3$ and $v_5$ are all in $A$.

Note that $v$ is chosen as a vertex of maximum degree in $A$, there must exist a vertex $z \in N(v)$ such that $z \nsim v_1$; otherwise, $d(v_1) > d(v)$, a contradiction. By Claim \ref{CLAIMT3}, $z$ is anticomplete to $\{x,y\}$. If $z \sim v_2$, then by Claim~\ref{CLAIMT3}, $N(z) \cap V(C_7)$ must equal to one of the following: $\{v_2, v_4, v_6\}$, $\{v_2, v_4, v_7\}$, or $\{v_2, v_5, v_7\}$. In the first case, $\{x, y, v_2, z, v_6, v_7\}$ induces a $P_2 \cup P_4$; in the latter two cases, $\{x, y, v_2, z, v_7, v_6\}$ induces a $P_2 \cup P_4$. Consequently, $z \nsim v_2$. Combined with $z \nsim v_1$, we conclude $N(z) \cap V(C_7) = \{v_3, v_5, v_7\}$.

The argument in the previous paragraph further shows that any vertex in $N(v)$ nonadjacent to $v_1$ must be adjacent to $v_5$. Similarly, since $v$ is chosen as a vertex of maximum degree in $A$, there exists a vertex $z^\prime \in N(v)$ such that $z^\prime \nsim v_5$. By Claim~\ref{CLAIMT3}, $z^\prime$ is anticomplete to $\{x,y,z\}$. By a similar analysis of $z^\prime$ as of $z$, we can deduce that $N(z^\prime) \cap V(C_7) = \{v_1, v_3, v_6\}$. But $G[\{x, y, z, v_7, v_6, z^\prime\}]$ is a $P_2\cup P_4$, a contradiction.

All three cases lead to a contradiction; this completes the proof of Theorem \ref{THM1}.
\end{proof}

\section{Perfect Divisibility of ($P_5$, bull)-Free Graphs}

In this section, we provide a new proof of Theorem \ref{THM2}.

\begin{proof}[\textnormal{\textbf{Proof of Theorem \ref{THM2}}}]

Seeking a contradiction, suppose $G$ is a ($P_5$, bull)-free MNPD graph. By minimality, $G$ is connected. $G[M(v)]$ cannot be perfect; otherwise, $G[M(v) \cup \{v\}]$ would be perfect and $\omega(G[N(v)]) < \omega(G)$, implying that $G$ is perfectly divisible, a contradiction. Consequently, by the Strong Perfect Graph Theorem~\cite{SPGT}, $G[M(v)]$ contains an odd antihole $X$. Since $G$ is connected and $v$ can be arbitrary, we may assume that $v$ has a neighbour $x$ such that $x$ has a neighbour in $V(X)$.

Let $V(X)=\{v_1,\ldots,v_{n}\}$, and $v_i\sim v_j$ if and only if $|i-j|\neq1$ (indices are modulo $n$). By Lemma \ref{LEM}, we may suppose that $N(x)\cap X = \{v_i\}$ or $N(x)\cap X=\{v_i, v_{i+1}\}$. Note that in both cases, $v-x-v_{i}-v_{i+2}-v_{i-1}$ is a $P_5$, a contradiction. This completes the proof of Theorem \ref{THM2}.
\end{proof}

\end{document}